\documentclass[11pt, a4paper]{article}
\usepackage{}
\usepackage{amsthm}
\usepackage{mathrsfs}
\usepackage{amsmath,amssymb,latexsym,color}
\usepackage[colorlinks,
linkcolor=blue,
anchorcolor=green,
citecolor=magenta
]{hyperref}
\usepackage{graphicx}
\usepackage{tikz}
\usetikzlibrary{calc}
\usepackage{CJK}

\usepackage{authblk}

\long\def\delete#1{}

\usepackage{color}

\definecolor{Blue}{rgb}{0,0,1}
\definecolor{Red}{rgb}{1,0,0}
\definecolor{DarkGreen}{rgb}{0,0.6,0}
\definecolor{DarkYellow}{rgb}{1,1,0.2}
\definecolor{DarkPurple}{rgb}{.6,0,1}

\usepackage{xcolor}
\usepackage[normalem]{ulem}

\usepackage{cleveref}
\crefformat{section}{\S#2#1#3}
\crefformat{subsection}{\S#2#1#3}
\crefformat{subsubsection}{\S#2#1#3}
\crefrangeformat{section}{\S\S#3#1#4 to~#5#2#6}
\crefmultiformat{section}{\S\S#2#1#3}{ and~#2#1#3}{, #2#1#3}{ and~#2#1#3}
\def\mmc{\mathscr{C}}
\def\ma{\mathscr{A}}

\def\mf{\mathscr{F}}
\def\mg{\mathscr{G}}

\def\ml{\mathscr{L}}
\def\mm{\mathscr{M}}
\def\mp{\mathscr{P}}

\def\ms{\mathscr{S}}

\def\bs{\setminus}
\def\mmp{\mathcal{P}_{m_1}}
\def\mmmp{\mathcal{P}_{m_2}}
\def\mmu{\mathscr{U}}

\def\ff{\mathbb{F}_q}

\def\ge{\geqslant}
\def\le{\leqslant}
\def\b{\brack}

\def\ro{\romannumeral}
\def\rmo{\rm{(1)}}
\def\rmt{\rm{(2)}}
\def\rmth{\rm{(3)}}

\numberwithin{equation}{section}

\newtheorem{thm}{Theorem}[section]
\newtheorem{lem}[thm]{Lemma}

\begin{document}
	\setcounter{page}{1}
	\renewcommand{\thefootnote}{}
	\newcommand{\remark}{\vspace{2ex}\noindent{\bf Remark.\quad}}
	\renewcommand{\abovewithdelims}[2]{%
		\genfrac{[}{]}{0pt}{}{#1}{#2}}

	%-------------------  First Head  -----------------------------------------
	
	\def\qed{\hfill$\Box$\vspace{11pt}}
	
	\title {\bf  Cross $t$-intersecting families for symplectic polar spaces}

	\author{Tian Yao\thanks{E-mail: \texttt{yaotian@mail.bnu.edu.cn}}}
	%\author{Benjian Lv\thanks{E-mail: \texttt{bjlv@bnu.edu.cn}}}
	\author{Kaishun Wang\thanks{Corresponding author. E-mail: \texttt{wangks@bnu.edu.cn}}}
	\affil{Laboratory of Mathematics and Complex Systems (Ministry of Education), School of
		Mathematical Sciences, Beijing Normal University, Beijing 100875, China}

	\date{}
	
	\openup 0.5\jot
	\maketitle

	\begin{abstract}
		Let $\mp$ be a symplectic polar space over a finite field $\ff$, and $\mp_m$ denote the collection of all $k$-dimensional totally isotropic subspace in $\mp$.
	Let $\mf_1\subset\mp_{m_1}$ and $\mf_2\subset\mp_{m_2}$ satisfy $\dim(F_1\cap F_2)\ge t$ for any $F_1\in\mf_1$ and $F_2\in\mf_2$. We say they are cross $t$-intersecting families.  Moreover, we say they are trivial if each member of them contains a fixed $t$-dimensional totally isotropic subspace. In this paper, we show that cross $t$-intersecting families with maximum product of sizes are trivial. We also describe the structure of non-trivial $t$-intersecting families with maximum product of sizes.
		
		\vspace{2mm}
		
		
		\
		
		
		%\footnotetext{ E-mail address: caomengyu@mail.bnu.edu.cN'(M.Cao), bjlv@bnu.edu.cN'(B.Lv), wangks@bnu.edu.cN'(K.Wang)}
		
	\end{abstract}
	
	\section{Introduction}
	
		Intersection problems originate from the famous Erd\H{o}s-Ko-Rado Theorem \cite{EKR}. In recent years, intersection problems for mathematical objects which are relative to vector spaces have been caught lots of attention \cite{VHM1,VHM2,AEKR,SPEKR1}.
	
	Let $n$ and $k$ be positive integers with $n\ge k$, $V$ an $n$-dimensional vector space over the finite field $\ff$, where $q$ is a prime power, and ${V\b k}_q$ denote the family of all $k$-dimensional subspaces of $V$. We usually replace ``$k$-dimensional subspace'' with ``$k$-subspace'' for short. Define the \emph{Gaussian binomial coefficient} by
	$${n\b k}_q:=\prod_{0\le i<k}\dfrac{q^{n-i}-1}{q^{k-i}-1},$$
	and set ${n\b0}_q=1$ and ${n\b c}=0$ if $c$ is a negative integer. Note that the size of ${V\b k}_q$ is ${n\b k}_q$. From now on, we will omit the subscript $q$.
	
	Let $t$ be a positive integer. A family $\mf\subset{V\b k}$ is called \emph{$t$-intersecting} if $\dim(F_1\cap F_2)\ge t$ for any $F_1,F_2\in\mf$. A $t$-intersecting family $\mf$ is called \emph{trivial} if there exists a $t$-subspace contained in each element of $\mf$. The Erd\H{o}s-Ko Rado Theorem for vector space \cite{VEKR2,VEKR1,VEKR3} shows that a $t$-intersecting subfamily of ${V\b k}$ with maximum size is trivial when $\dim V>2k$. The structure of non-trivial $t$-intersecting subfamily of ${V\b k}$ with maximum size was determined via the parameter ``$t$-covering number", see \cite{VHM1,VHM2}. For $\mf_1\in{V\b k_1}$ and $\mf_2\subset{V\b k_2}$, we say they are cross $t$-intersecting if $\dim(F_1\cap F_2)$ holds for any $F_1\in\mf_1$ and $F_2\in\mf_2$. Recently, Cao et al \cite{CAO} describe the structure of cross $t$-intersecting families with the first and second larges product of sizes.

	Let $f$ be a non-degenerate alternating bilinear form defined on a $2\nu$-dimensional vector space $\ff^{2\nu}$ over $\ff$.  An $m$-subspace $M$ of $V$ is called \emph{totally isotropic} if $f(x,y)=0$ holds for any $x,y\in M$. We know that $\nu$ is the dimension of the maximal totally isotropic subspaces. Denote the sets of all totally isotropic subspaces and $m$-dimensional totally isotropic subspaces with respect to $f$ by $\mp$ and $\mp_m$, respectively, where $0\le m\le\nu$. Equipped with the inclusion relation, $\mp$ is a \emph{symplectic polar space}, denoted by the same symbol $\mp$. The \emph{rank} of $\mp$ is the dimension of the maximal totally isotropic subspaces. The symplectic space is one of the six kinds of classical polar spaces with $\nu\ge2$ \cite{CPS}.

	A subfamily of $\mp_m$ is called \emph{$t$-intersecting} if any two members have a intersection with dimension at least $t$. The maximum sized $t$-intersecting subfamilies of $\mp_m$ were widely studied and described. See \cite{SPEKR2,SPEKR1} for $t=1$ and \cite{SPEKRT} for all $t$. Recently, the authors characterized the second largest $t$-intersecting families \cite{YT}. There are also some results for other classical polar spaces, see\cite{012,PLANE,HMQ,MENTION1,SPEKR2,SPEKR1} for more details.
	
	Let $\mf_1\subset\mp_{m_1}$ and $\mf_2\subset\mp_{m_2}$ satisfy that $\dim(F_1\cap F_2)\ge t$ for any $F_1\in\mf_1$ and $F_2\in\mf_2$. 
	We say they are \emph{cross $t$-intersecting}. 
	Moreover, they are called \emph{trivial} if each member of them contains a fixed $t$-dimensional totally isotropic subspace. 
	In \cite{sp-add},  Ihringer characterized the structure of cross $t$-intersecting subfamilies of $\mp_\nu$ with maximum product of sizes.
	
	The first main result of this paper is the following.
	
	\begin{thm}\label{trivial}
		Let $\nu$, $m_1$, $m_2$ and $t$ be positive integers with $m_1\ge m_2\ge t$ and $2\nu\ge2m_1+m_2+1$, and $\mp$  a symplectic polar spaces with rank $\nu$. Suppose $\mf_1\subset\mp_{m_1}$ and $\mf_2\subset\mp_{m_2}$ are cross $t$-intersecting families with maximum product of sizes. Then there exists a $t$-dimensional totally isotropic subspace contained in each member of $\mf_1$ and $\mf_2$. 
	\end{thm}

	Based on Theorem \ref{trivial}, we get a more general theorem, see Theorem \ref{trivial2}.
	
	For a subspace $A$ of $\ff^{2\nu}$ and a positive integer $a$, write $\mm(a,A)=\{F\in\mp_a: F\subset A\}$. Let $M\in\mp_{m_2+1}$,  $T\in\mm(t,M)$ and $S\in\mp_{t+1}$. 
	Write
	\begin{equation}\label{c1-c4}
		\begin{aligned}
			&\mmc_1(M,T;m_1,t)=\{F\in\mp_{m_1}: T\subset F, \dim(F\cap M)\ge t+1\},\\
			&\mmc_2(M,T;m_2)=\{F\in\mp_{m_2}: T\subset F\}\cup{M\b m_2},\\
			&\mmc_3(S;m_1)=\{F\in\mp_{m_1}:S\subset F\},\\
			&\mmc_4(S;m_2,t)=\{F\in\mp_{m_2}:\dim(F\cap S)\ge t\}.
		\end{aligned}
	\end{equation}
Observe that $\mmc_1(M,T;m_1,t)$ and $\mmc_2(M,T;m_2)$ are cross $t$-intersecting families. So are $\mmc_3(S;m_1)$ and $\mmc_4(S;m_2,t)$.	 
Our second main result describe the structure of cross $t$-intersecting families with the second largest product of sizes.	

\begin{thm}\label{non-trivial}
	Let $\nu$, $m_1$, $m_2$ and $t$ be positive integers with $m_1\ge m_2\ge t+1$ and $2\nu\ge2m_1+m_2+t+3$, and $\mp$  a symplectic polar spaces with rank $\nu$.	Suppose that $\mf_1\subset\mp_{m_1}$ and $\mf_2\subset\mp_{m_2}$ are non-trivial cross $t$-intersecting families with maximum product of sizes. 
	\begin{itemize}
		\item[\rmo] If $m_2>2t$ or $(m_1,m_2,t)=(2,2,1),(3,2,1)$, then there exist $M\in\mp_{m_2+1}$ and $T\in\mm(t,M)$ such that
		\begin{itemize}
			\item[\rm{(\ro1)}] $\mf_1=\mmc_1(M,T;m_1,t)$, $\mf_2=\mmc_2(M,T;m_2)$; or
			\item[\rm{(\ro2)}] $m_1=m_2$ and $\mf_1=\mmc_2(M,T;m_1)$, $\mf_2=\mmc_1(M,T;m_2,t)$.
		\end{itemize}
		\item[\rmt] If $m_2\le2t$ and $(m_1,m_2,t)\neq(2,2,1),(3,2,1)$, then there exists $S\in\mp_{t+1}$ such that
		\begin{itemize}
			\item[\rm{(\ro1)}] $\mf_1=\mmc_3(S;m_1)$, $\mf_2=\mmc_4(S;m_2,t)$; or
			\item[\rm{(\ro2)}] $m_1=m_2$ and $\mf_1=\mmc_4(S;m_1,t)$, $\mf_2=\mmc_3(S;m_2)$.
		\end{itemize}
	\end{itemize}
\end{thm}

\section{Some Lemmas}

In this section, we give some useful lemmas in preparation for the proof of Theorems \ref{trivial} and \ref{non-trivial}.
\begin{lem}\label{gaosifangsuo}
	Let $m$ and $i$ be positive integers with $i<m$. Then the following hold.
	\begin{itemize}
		\item[\rmo] $q^{m-i}<\frac{q^m-1}{q^i-1}<q^{m-i+1}$ and $q^{i-m-1}<\frac{q^i-1}{q^m-1}<q^{i-m}$;
		\item[\rmt] $q^{i(m-i)}<{m\b i}<q^{i(m-i+1)}$.
	\end{itemize}
\end{lem}

For $1\le m\le \nu-1$, let $Q$ be an $m$-subspace of $\ff^{2\nu}$ and $\alpha_1,\dots,\alpha_m$ any basis of $Q$. Note that the rank of the matrix $(f(\alpha_i,\alpha_j))_{m\times m}$ is even and independent of the choice of the basis. We say $Q$ is of  \emph{type $(m,s)$} if the rank of the matrix $(f(\alpha_i,\alpha_j))_{m\times m}$ is $2s$. Note that $Q$ is of type $(m,0)$ if and only if $Q$ is totally isotropic.
	
For positive integers $m_1$ and $ m$ with $m_1\le m\le\nu$, let $N'(m_1;m;2\nu)$ be the number of members of $\mp_{m}$ containing a fixed member of $\mp_{m_1}$. By \cite[Theorem 1]{JS1}, the size of $\mp_m$ is ${\nu\b m}\prod_{i=0}^{m-1}(q^{\nu-i}+1)$, from which we derive that
\begin{equation*}\label{JS}
	N'(m_1;m;2\nu)=\prod\limits_{i=1}^{m-m_1}\dfrac{q^{2(\nu-m+i)}-1}{q^{i}-1}.
\end{equation*}
By \cite[Theorem 9 in Chapter 2]{JS2}, the number of members of $\mp_m$ contained in a fixed $(m+1,1)$-type subspace is $q+1$.	
	
Let $\mf\subset\mp_m$. For $T\in\mp$, if $\dim(T\cap F)\ge t$ holds for each $F\in\mf$, we say $T$ is a \emph{$t$-cover} of $\mf$. Let $\tau_t(\mf)$ denote the minimum dimension of $\mf$'s $t$-covers. From \cite[Lemma 2.4]{CAO}, we derive the following Lemma.	
	
\begin{lem}\label{sp-fangsuo}
	Let $\nu$, $m$, $s$ and $t$ be positive integers with  $\nu>m\ge s\ge t$. Suppose $\mf\subset\mp_k$, $X$ is a $t$-cover of $\mf$  with dimension $x$ and $S\in\mp_s$. If $\dim(X\cap S)=y<t$, then there exists $R\in\mp_{s+t-y}$ such taht $S\subset R$ and
	$$|\mf_S|\le{x-t+1\b1}^{t-y}|\mf_R|.$$
\end{lem}	

\begin{lem}\label{sp-yibanshangjie}
	Let $\nu$, $m_1$, $m_2$ and $t$ be positive integers with $\nu>m_i\ge t$, $i\in\{1,2\}$,  and $2\nu\ge2m_1+m_2-t$. Suppose $\mf\subset\mmp$ and $\mg\subset\mmmp$ are cross $t$-intersecting. Then
	$$|\mf|\le{\tau_t(\mf)\b t}{m_2-t+1\b1}^{\tau_t(\mg)-t}N'(\tau_t(\mg);m_1;2\nu).$$
\end{lem}
	\begin{proof}
	Let $S$ be a $t$-cover of $\mf$ with dimension $\tau_t(\mf)$. From
	\begin{equation}\label{sp-yibanshangjie-f1}
		\mf=\bigcup\limits_{W\in{S\b t}}\mf_W,
	\end{equation} 
	we get
	$$|\mf|\le{\tau_t(\mf)\b t}N'(t;m_1;2\nu),$$
	which implies that the desired result holds for $\tau_t(\mg)=t$. In the following, we assume that $\tau_t(\mg)>t$.

	Let $W_1\in{S\b t}$ with $\mf_{W_1}\neq\emptyset$. We first give an upper bound for $|\mf_{W_1}|$.
Since $\tau_t(\mg)>t$, there exists $G\in\mg$ such that $\dim(G\cap W_1)<t$. Notice that $G$ is a $t$-cover of $\mf$. By Lemma \ref{sp-fangsuo}, there exists a $(2t-\dim(W_1\cap G))$-dimensional totally isotropic subspace  $W_2$ such that
$$|\mf_{W_1}|\le{m_2-t+1\b1}^{\dim W_2-\dim W_1}|\mf_{W_2}|.$$
By $|\mf_{W_1}|>0$, we have $|\mf_{W_2}|>0$, which implies that $\dim W_2\le m_1$.
If $\dim W_2<\tau_t(\mg)$, there exists $G'\in\mg$ with $\dim(W_2\cap G')<t$. 
Using Lemma \ref{sp-fangsuo} repeatedly, we get a series of totally isotropic subspaces $W_1,W_2,\dots,W_u$ with $\dim W_{u-1}<\tau_t(\mg)\le\dim W_u\le m_1$ and 
$$|\mf_{W_i}|\le{m_2-t+1\b1}^{\dim W_{i+1}-\dim W_i}|\mf_{W_{i+1}}|$$
for each $i\in\{1,\dots,u-1\}$.
Hence
$$|\mf_{W_1}|\le{m_2-t+1\b1}^{\dim W_u-t}|\mf_{W_u}|\le{m_2-t+1\b1}^{\dim W_u-t}N'(\dim W_u;m_1;2\nu).$$
From Lemma \ref{gaosifangsuo} and $2\nu\ge2m_1+m_2-t$, for each $a\in\{0,\dots,m_1-1\}$, we obtain
$$\dfrac{N'(a;m_1;2\nu)}{N'(a+1;m_1;2\nu)}=\dfrac{q^{2(\nu-a)}-1}{q^{m_1-a}-1}\ge q^{2\nu-2m_1+1}\ge q^{m_2-t+1}\ge{m_2-t+1\b1}.$$
Note that $\dim W_u\ge\tau_t(\mg)$. We have
\begin{equation*}
	\begin{aligned}
		|\mf_{W_1}|\le{m_2-t+1\b1}^{\dim W_u-t}N'(\dim W_u;m_1;2\nu)\le{m_2-t+1\b1}^{\tau_t(\mg)-t}N'(\tau_t(\mg);m_1;2\nu).
	\end{aligned}
\end{equation*}
Together with \eqref{sp-yibanshangjie-f1}, we get
$$|\mf|\le\sum_{W\in{S\b t}}|\mf_W|\le{\tau_t(\mf)\b t}{m_2-t+1\b1}^{\tau_t(\mg)-t}N'(\tau_t(\mg);m_1;2\nu),$$
as desired.
\end{proof}

\section{Proof of Theorem \ref{trivial}}

To prove Theorem \ref{trivial}, we need the following two lemmas.
\begin{lem}\label{sp-dijian}
	Let $\nu$, $b$, $c$ and $t$ be positive integers with $2\nu\ge 2b+c+1$ and $\nu>b\ge t+1$, $c\ge t$. For $x\in\{t,\dots,b\}$, let
	$$g_{b,c}(x)={x\b t}{c-t+1\b1}^{x-t}N'(x;b;2\nu).$$
	Then $g_{b,c}(x)$ is decreasing with respect to $x$.
\end{lem}	
\begin{proof}
	By Lemma \ref{gaosifangsuo} and $2\nu\ge 2b+c+1$, for each $x\in\{t,\dots,b-1\}$, we have
	$$\dfrac{g_{b,c}(x+1)}{g_{b,c}(x)}=\dfrac{(q^{x+1}-1)(q^{c-t+1}-1)(q^{b-x}-1)}{(q^{x-t+1}-1)(q-1)(q^{2(\nu-x)}-1)}<q^{2b+c+1-2\nu}\le1.$$
	Then $g_{b,c}(x+1)<g_{b,c}(x)$, as desired.
\end{proof}

\begin{lem}\label{sp-cover}
	Let $\nu$, $m_1$, $m_2$ and $t$ be positive integers with $\nu>m_i\ge t$, $i\in\{1,2\}$. Suppose $\mf_1\subset\mp_{m_1}$ and $\mf_2\subset\mp_{m_2}$ are cross $t$-intersecting. For each $i\in\{1,2\}$, let $\ms_i$ denote the set of all $t$-covers of $\mf_i$ with dimension $\tau_t(\mf_i)$. Then $\ms_1$ and $\ms_2$ are cross $t$-intersecting.
\end{lem}
\begin{proof}
	Let $S_1\in\ms_1$ and $S_2\in\ms_2$. It is sufficient to show that $\dim(S_1\cap S_2)\ge t$.	
	
Since $\mf_1$ and $\mf_2$ are cross $t$-intersecting, we have $\dim S_2=\tau_t(\mf_2)\le m_1<\nu$. 	
Then there exist two $(\tau_t(\mf_2)+1)$-dimensional totally isotropic subspaces $Y_1$ and $Y_2$ containing $S_2$  such that $S_2=Y_1\cap Y_2$ and there exists no maximal totally isotropic subspace contains both of them. 	
Therefore, there exists $k\in\{1,2\}$ such that $Y_k\cap S_1=S_2\cap S_1$. Similarly, it is routine to check that there exist $F_1\in\mmp$ and $F_2\in\mmmp$ such that $S_1\subset F_2$, $S_2\subset F_1$ and $F_1\cap F_2=S_1\cap S_2$.	
By the maximality of $\mf_1$ and $\mf_2$, since $\mf_1\cup\{F_1\}$ and $\mf_2$ are still cross $t$-intersecting, we have $F_1\in\mf_1$. Similarly, we have $F_2\in\mf_2$.
Thus $\dim(S_1\cap S_2)=\dim(F_1\cap F_2)\ge t$, as desired.
\end{proof}	

\begin{proof}[\bf Proof of Theorem \ref{trivial}]
	
	Suppose $\mf_1\subset\mp_{m_1}$ and $\mf_2\subset\mp_{m_2}$ are cross $t$-intersecting families.
	Assume that $\tau_t(\mf_1)=\tau_t(\mf_2)=t$. Let $T_1$ and $T_2$ be $t$-covers of $\mf_1$ and $\mf_2$ with dimension $t$, respectively. By Lemma \ref{sp-cover} , we have $T_1=T_2:=T$.
	Then 
	$$|\mf_1|\le N'(t;m_1;2\nu),\quad|\mf_2|\le N'(t;m_2;2\nu),$$
	and two equalities hold at the same time if and only if $\mf_i=\{F\in\mp_{m_i}:  T\subset F\}$ for each $i\in\{1,2\}$.

To finish our proof, it is sufficient to show 
\begin{equation}\label{sp-f1f2<nn}
	|\mf_1||\mf_2|<N'(t;m_1;2\nu)N'(t;m_2;2\nu)
\end{equation}
if $(\tau_t(\mf_1),\tau_t(\mf_2))\neq(t,t)$. By Lemma \ref{sp-yibanshangjie} and $2\nu\ge2m_1+m_2+1$, $m_1\ge m_2$, we have
\begin{equation*}
	\begin{aligned}
		|\mf_1||\mf_2|\le&\left({\tau_t(\mf_1)\b t}{m_1-t+1\b1}^{\tau_t(\mf_1)-t}N'(\tau_t(\mf_1);m_2;2\nu)\right)\\
		&\cdot\left({\tau_t(\mf_2)\b t}{m_2-t+1\b1}^{\tau_t(\mf_2)-t}N'(\tau_t(\mf_2);m_1;2\nu)\right).
	\end{aligned}
\end{equation*}
Note that $t\le\tau_t(\mf_1)\le m_2$ and $t\le\tau_t(\mf_2)\le m_1$.
Together with $(\tau_t(\mf_1),\tau_t(\mf_2))\neq(t,t)$ and Lemma \ref{sp-dijian}, \eqref{sp-f1f2<nn} follows.	
\end{proof}

Based on Theorem \ref{trivial},  we obtain a more general theorem.

\begin{thm}\label{trivial2}
	Let $d$, $\nu$, $t$, $m_1$,\dots, $m_d$ be positive integers with $d\ge2$, $m_1\ge m_2\ge\cdots\ge m_d\ge t$ and $2\nu\ge 2m_1+m_2+1$. If $\mf_1\subset\mmp$,\dots, $\mf_d\subset\mp_{m_d}$ satisfy that $\dim(F_1\cap\cdots\cap F_d)\ge t$ 
	for any $F_i\in\mf_i$, $i=1,\dots,d$. If $\prod_{i=1}^d|\mf_i|$ reaches to the maximum value, then there exists  $T\in\mp_t$ such that $\mf_i=\{F\in\mp_{m_i}: T\subset F\}$ for each $i\in\{1,\dots,d\}$.
\end{thm}
\begin{proof}
	For distinct $i,j\in\{1,\dots,d\}$, $\mf_i$ and $\mf_j$ are cross $t$-intersecting families. Then by
	Theorem \ref{trivial}, we have
	$$|\mf_i||\mf_j|\le N'(t;m_i;2\nu)N'(t;m_j;2\nu).$$
	Therefore
\begin{equation}\label{sp-cross-prod}
	\begin{aligned}
		\left(\prod_{s=1}^d|\mf_s|\right)^{d-1}&=\prod_{1\le i<j\le d}|\mf_i||\mf_j|\\
		&\le\prod_{1\le i<j\le d}N'(t;m_i;2\nu)N'(t;m_j;2\nu)\\
		&=\left(\prod_{s=1}^dN'(t;m_s;2\nu)\right)^{d-1},
	\end{aligned}
\end{equation}
and equality holds if and only if $|\mf_i||\mf_j|=N'(t;m_i;2\nu)N'(t;m_j;2\nu)$ for any distinct $i,j\in\{1,\dots,d\}$.

	Note that the product of sizes of families $\{F\in\mp_{m_i}: S\subset F\}$, $i=1,\dots,d$, where $S\in\mp_t$, reaches to the upper bound of \eqref{sp-cross-prod}. 
Therefore, by assumption and Theorem \ref{trivial}, for distinct $i,j\in\{1,\dots,d\}$, there exists $T_{i,j}\in\mp_t$ such that 
$$\mf_i=\{F\in\mp_{m_i}:  T_{i,j}\subset F\},\quad\mf_j=\{F\in\mp_{m_j}:  T_{i,j}\subset F\}.$$
If there exists $j'\in\{1,2,\dots,d\}$ such that $T_{i,j'}\neq T_{i,j}$, we have
$$\mf_i\subset\{F\in\mp_i:  T_{i,j}+T_{i,j'}\subset F\}.$$
	Together with $2\nu\ge2m_1+m_2+1$, $m_1\ge m_i$ and $\dim(T_{i,j}+T_{i,j'})\ge t+1$, we get
$$N'(t+1;m_i;2\nu)<N'(t;m_i;2\nu)=|\mf_i|\le N'(t+1;m_i;2\nu),$$
a contradiction.  Therefore, there exists $T\in\mp_t$ such that $T_{i,j}=T$ for any distinct $i,j\in\{1,\dots,d\}$. Then the desired result follows.
\end{proof}

\section{Proof of Theorem \ref{non-trivial}}	

Let $\nu>m_i\ge t+1$, $i\in\{1,2\}$. 
Suppose $M\in\mp_{m_2+1}$, $T\in\mm(t,M)$ and  $S\in\mp_{t+1}$. Let 	$\mmc_1(M,T;m_1,t)$, 
$\mmc_2(M,T;m_2)$, 
$\mmc_3(S;m_1)$ and
$\mmc_4(S;m_2,t)$ are families defined in \eqref{c1-c4}. By \cite[Theorem 3.11]{JS}, $|\mmc_1(M,T;m_1,t)|\cdot|\mmc_2(M,T;m_2)|$ and $|\mmc_3(S;m_1)|\cdot|\mmc_4(S;m_2,t)|$ are independent on the choice of $M$, $T$ and $S$. 	

Write
\begin{equation*}
	\begin{aligned}
		c_1(\nu,m_1,m_2,t)&=|\mmc_1(M,T;m_1,t)|\cdot|\mmc_2(M,T;m_2)|,\\
		c_2(\nu,m_1,m_2,t)&=|\mmc_3(S;m_1)|\cdot|\mmc_4(S;m_2,t)|.
	\end{aligned}
\end{equation*}	
It is routine to check that
\begin{equation}\label{cp3cp4-size}
	\dfrac{c_2(\nu,m_1,m_2,t)}{N'(t+1;m_1;2\nu)}={t+1\b1}N'(t;m_2;2\nu)-q{t\b1}N'(t+1;m_2;2\nu).
\end{equation}
	In the following, we show some inequalities for $c_1(\nu,m_1,m_2,t)$ and $c_2(\nu,m_1,m_2,t)$.

	For $S\in\mp_s$, $T\in{S\b t}$ and $j\in\{t,t+1,\dots,s\}$, write
\begin{equation*}\label{xla}
	\begin{aligned}
		\ml_j(S,T;m)&=\{(I,F)\in\mp_j\times\mp_m:  T\subset I\subset S, I\subset F\},\\
		\ma_j(S,T;m)&=\{F\in\mp_m:  T\subset F, \dim(F\cap S)=j\}
	\end{aligned}
\end{equation*}
and
$$s_0(\nu,m,s,t)={s-t\b1}N'(t+1;m;2\nu)-q{s-t\b2}N'(t+2;m;2\nu).$$

	\begin{lem}\label{c1>c0}
	Let $\nu$, $m_1$, $m_2$ and $t$ be positive integers with $\nu>m_i\ge t+1$, $i\in\{1,2\}$. 
	Then 
	$$c_1(\nu,m_1,m_2,t)>s_0(\nu,m_1,m_2+1,t)N'(t;m_2;2\nu).$$ 
	Moreover, if $2\nu\ge m_1+2m_2-t+4$,then
	$$c_1(\nu,m_1,m_2,t)>\left({m_2-t+1\b1}-q^{-2}\right)N'(t+1;m_1;2\nu)N'(t;m_2;2\nu).$$
\end{lem}
	\begin{proof}
	Let $M\in\mp_{m_2+1}$ and $T\in{M\b t}$. 
	For each $j\in\{t+1,\dots,m_2+1\}$, by double counting $|\ml_j(M,T;m_1)|$, we have
	\begin{equation}\label{sp-suanliangci-formu}
		|\ml_j(M,T;m_1)|={m_2-t+1\b j-t}N'(j;m_1;2\nu)=\sum_{i=j}^{m_2+1}{i-t\b j-t}|\ma_i(M,T;m_1)|.
	\end{equation}
Then
\begin{equation*}
	\begin{aligned}
		s_0(\nu,m_1,m_2+1,t)=&\ |\ml_{t+1}(M,T;m_1)|-q|\ml_{t+2}(M,T;m_1)|\\
		=&\sum_{i=t+1}^{m_2+1}{i-t\b1}|\ma_i(M,T;m_1)|-q\sum_{i=t+2}^{m_2+1}{i-t\b2}|\ma_i(M,T;m_1)|\\
		=&\sum_{i=t+1}^{t+2}|\ma_i(M,T;m_1)|+\sum_{i=3}^{m_2+1}\left({i-t\b1}-q{i-t\b2}\right)|\ma_i(M,T;m_1)|.
	\end{aligned}
\end{equation*}
	Note that ${i-t\b1}<q{i-t\b2}$ for $i\ge t+3$. Then
\begin{equation*}
	\begin{aligned}
		s_0(\nu,m_1,m_2+1,t)&\le|\ma_{t+1}(M,T;m_1)|+|\ma_{t+2}(M,T;m_1)|\\
		&=|\{F\in\mmp:  T\subset F, \dim(F\cap M)\in\{t+1,t+2\}\}|\\
		&\le|\{F\in\mmp:  T\subset F, \dim(F\cap M)\ge t+1\}|\\
		&\le|\mmc_1(M,T;m_1,t)|,
	\end{aligned}
\end{equation*}
from which we get  $c_1(\nu,m_1,m_2,t)>s_0(\nu,m_1,m_2+1,t)N'(t;m_2;2\nu)$. 

	Together with Lemma \ref{gaosifangsuo} and $2\nu\ge m_1+2m_2-t+4$, we have
%\begin{equation*}
\begin{align*}
	\dfrac{c_1(\nu,m_1,m_2,t)}{N'(t;m_2;2\nu)}&>s_0(\nu,m_1,m_2+1,t)\\
	&=\left({m_2-t+1\b1}-q{m_2-t+1\b2}\dfrac{q^{m_1-t-1}-1}{q^{2(\nu-t-1)}-1}\right)N'(t+1;m_1;2\nu)\\
	&\ge\left({m_2-t+1\b1}-q^{m_1+2m_2-t+2-2\nu}\right)N'(t+1;m_1;2\nu)\\
	&\ge\left({m_2-t+1\b1}-q^{-2}\right)N'(t+1;m_1;2\nu),
\end{align*}
%\end{equation*}
as desired.
\end{proof}

\begin{lem}\label{cp1-cp2}
	Let $\nu$, $m_1$, $m_2$ and $t$ be positive integers with $m_1\ge m_2\ge t+1$ and $2\nu\ge2m_1+m_2+t+2$.
	\begin{itemize}
		\item[\rmo] If $m_2>2t$ or $(m_1,m_2,t)=(2,2,1),(3,2,1)$, then $c_1(\nu,m_1,m_2,t)>c_2(\nu,m_1,m_2,t)$.
		\item[\rmt] If $m_2\le 2t$ and $(m_1,m_2,t)\neq(2,2,1),(3,2,1)$, then $c_1(\nu,m_1,m_2,t)<c_2(\nu,m_1,m_2,t)$.
	\end{itemize}
\end{lem}
\begin{proof}
	Write $c_3(\nu,m_1,m_2,t)=c_2(\nu,m_1,m_2,t)-c_1(\nu,m_1,m_2,t)$.
	
	(1) Suppose $m_2>2t$. By $2\nu\ge2m_1+m_2+t+2\ge m_1+2m_2-t+4$, \eqref{cp3cp4-size} and Lemma \ref{c1>c0}, we have
\begin{equation*}
	\begin{aligned}
		\dfrac{c_1(\nu,m_1,m_2,t)}{N'(t;m_2;2\nu)}>\ q{m_2-t\b1}N'(t+1;m_1;2\nu)>{t+1\b1}N'(t+1;m_1;2\nu)>\dfrac{c_2(\nu,m_1,m_2,t)}{N'(t;m_2;2\nu)}.
	\end{aligned}
\end{equation*}
which implies that $c_1(\nu,m_1,m_2,t)>c_2(\nu,m_1,m_2,t)$.	
	
	Assume  that $(m_1,m_2,t)=(2,2,1)$.
By \eqref{cp3cp4-size}, it is routine to check that
$$c_3(\nu,2,2,1)=\left({2\b1}N'(1;2;2\nu)-q\right)-{2\b1}\left(N'(1;2;2\nu)+q^2\right)<0.$$
Then the desired result follows.	

When $m_1\ge3$, it is routine to check that
\begin{equation*}
	\begin{aligned}
		c_1(\nu,m_1,2,1)&=\left({2\b1}N'(2;m_1;2\nu)-qN'(3;m_1;2\nu)\right)(N'(1;2;2\nu)+q^2),\\
		c_2(\nu,m_1,2,1)&=N'(2;m_1;2\nu)\left({2\b1}N'(1;2;2\nu)-q\right).
	\end{aligned}
\end{equation*}
Then
\begin{equation}\label{sp-m21}
	\dfrac{c_3(\nu,m_1,2,1)}{qN'(2;m_1;2\nu)}=\left(\dfrac{q^{2\nu-2}-1}{q-1}+q^2\right)\dfrac{q^{m_1-2}-1}{q^{2\nu-4}-1}-q^2-q-1.
\end{equation}
Suppose $(m_1,m_2,t)=(3,2,1)$. By $2\nu\ge2m_1+m_2=8$,  \eqref{sp-m21} and Lemma \ref{gaosifangsuo}, we have
\begin{equation*}
	\begin{aligned}
		\dfrac{c_3(\nu,3,2,1)}{qN'(2;3;2\nu)}&<\dfrac{q^{2\nu-2}-1}{q^{2\nu-4}-1}-q^2-q=\dfrac{q^2-1}{q^{2\nu-4}-1}-q<q^{-2}-q<0.
	\end{aligned}
\end{equation*}
Then the desired result follows.

	(2) Suppose $m_2=2t$ and $t=1$. By assumption, we have $m_1\ge4$. From Lemma \ref{gaosifangsuo}, $2\nu\ge2m_1+m_2+t+2$ and \eqref{sp-m21}, we obtain
\begin{equation*}
	\begin{aligned}
		\dfrac{c_3(\nu,m_1,2,1)}{qN'(2;m_1;2\nu)}&>q^2\cdot q^{m_1-3}+q^2\cdot q^{m_1+1-2\nu}-q^2-q-1\\
		&=q^{m_1-1}+q^{m_1+3-2\nu}-q^2-q-1\\
		&>q^3-q^2-q-1\\
		&>0.
	\end{aligned}
\end{equation*}
Then $c_1(\nu,m_1,2,1)<c_2(\nu,m_1,2,1)$.

Suppose $m_2=2t$ and $t\ge2$. We have $m_2\ge t+2$. Let $M\in\mp_{m_2+1}$ and $T\in{M\b t}$.
From \eqref{sp-suanliangci-formu}, we obtain
\begin{equation*}
	\begin{aligned}
		{t+1\b1}N'(t+1;m_1;2\nu)&=\sum_{j=t+1}^{2t+1}{j-t\b1}|\ma_j(M,T;m_1)|\\
		&=|\mmc_1(M,T;m_1,t)|+\sum_{j=t+2}^{2t+1}q{j-t-1\b1}|\ma_j(M,T;m_1)|.
	\end{aligned}
\end{equation*}
Set
$$\alpha=\sum_{j=t+2}^{2t+1}q{j-t-1\b1}|\ma_j(M,T;m_1)|.$$
We have $\alpha\ge q|\ma_{t+2}(M,T;m_1)|$ and
$$c_1(\nu,m_1,2t,t)=\left({t+1\b1}N'(t+1;m_1;2\nu)-\alpha\right)\left(N'(t;2t;2\nu)+q^{t+1}{t\b1}\right).$$	
Together with \eqref{cp3cp4-size}, we get
\begin{equation}\label{sp-ca1}
	\dfrac{c_3(\nu,m_1,2t,t)}{N'(t+1;m_1;2\nu)}>\dfrac{\alpha N'(t;2t;2\nu)}{N'(t+1;m_1;2\nu)}-q{t\b1}N'(t+1;2t;2\nu)-q^{t+1}{t\b1}{t+1\b1}.
\end{equation}	
By Lemma \ref{gaosifangsuo}, we have
\begin{equation}\label{sp-ca2}
	N'(t+1;m_1;2\nu)\le\prod_{i=1}^{m_1-t-1}q^{2(\nu-m_1)+i+1}=q^{2(m_1-t-1)(\nu-m_1)+\frac{(m_1-t-1)(m_1-t+2)}{2}}.
\end{equation}
Assume that $m_1=2t$. By Lemma \ref{gaosifangsuo} and \cite[Theorem 2.10]{QY}, we have
\begin{align*}
	\alpha\ge&\ q|\ma_{t+2}(M,T;m_1)|\\
	%=&\ q{t+1\b2}\cdot\dfrac{N(t+2;2t+1,m_1;2\nu)}{{2t+1\b t+2}}\\
	%\ge&\ q{t+1\b2}\cdot\dfrac{p(m_1-t-2,t-2;t+2,2t+1,m_1;2\nu)}{{2t+1\b t+2}}\\
	\ge&\ q{t+1\b2}\cdot\left(q^{\frac{(t-2)(t-1)}{2}+(t-2)(\nu-2t-1)+(t-2)(\nu-2t)}{t-1\b1}\right)\\
	%\ge&\ q^{2(t-2)(\nu-2t)+\frac{t^2-3t+4}{2}}{t+1\b2}.
	\ge&\ q^{2(t-2)(\nu-2t)+\frac{t^2-t+2}{2}}{t\b1}.
\end{align*}
Together with \eqref{sp-ca1}, \eqref{sp-ca2}  and $2\nu\ge2m_1+m_2=6t$, we obtain
\begin{align*}
	\dfrac{c_3(\nu,2t,2t,t)}{{t\b1}N'(t+1;2t;2\nu)}&>\dfrac{q^{2(t-2)(\nu-2t)+\frac{t^2-t+2}{2}}(q^{2(\nu-t)}-1)}{q^t-1}-q^{2(t-1)(\nu-2t)+\frac{t^2+t}{2}}-q^{2t+2}\\
	&\ge\ q^{2(t-1)(\nu-2t)+\frac{t^2+t+2}{2}}-q^{2(t-1)(\nu-2t)+\frac{t^2+t}{2}+1}\\
	&=\ 0.
\end{align*}
Now assume that $m_1\ge 2t+1$. By Lemma \ref{gaosifangsuo} and \cite[Theorem 2.10]{QY}, we have
\begin{align*}
	\alpha\ge&\ q|\ma_{t+2}(M,T;m_1)|\\
	%=&\ q{t+1\b2}\cdot\dfrac{N(t+2;2t+1,m_1;2\nu)}{{2t+1\b t+2}}\\
	%\ge&\ q{t+1\b2}\cdot\dfrac{p(m_1-t-2,t-1;t+2,2t+1,m_1;2\nu)}{{2t+1\b t+2}}\\
	\ge&\ q{t+1\b2}\cdot q^{\frac{(t-1)t}{2}+\frac{(m_1-2t-1)(m_1-2)}{2}+2(m_1-t-2)(\nu-m_1)}\\
	%\ge&\ q^{2(t-2)(\nu-2t)+\frac{t^2-3t+4}{2}}{t+1\b2}.
	\ge&\ q^{2(m_1-t-2)(\nu-m_1)+\frac{t^2+t}{2}+\frac{(m_1-2t-1)(m_1-2)}{2}}{t\b1}.
\end{align*}
Together with \eqref{sp-ca1}, \eqref{sp-ca2}, $2\nu\ge2m_1+m_2+t\ge7t+2$ and
$$N'(t;2t;2\nu)\ge\prod_{i=1}^tq^{2(\nu-2t)+i}=q^{2t(\nu-2t)+\frac{t(t+1)}{2}}\ge q^{4\nu-8t+3},$$
we get
\begin{align*}
	\dfrac{c_3(\nu,m_1,2t,t)}{{t\b1}N'(t+1;m_1;2\nu)}&>\left(\dfrac{\alpha}{{t\b1}N'(t+1;m_1;2\nu)}-\dfrac{q(q^t-1)}{q^{2(\nu-t)}-1}-\dfrac{q^{2t+2}}{q^{4\nu-8t+3}}\right)N'(t;2t;2\nu)\\
	%&\ge\left(\dfrac{q^{2(m_1-t-2)(\nu-m_1)+\frac{t^2+t}{2}+\frac{(m_1-2t-1)(m_1-2)}{2}}}{q^{2(m_1-t-1)(\nu-m_1)+\frac{(m_1-t-1)(m_1-t+2)}{2}}}-q^{3t+1-2\nu}-q^{10t-1-4\nu}\right)N'(t;2t;2\nu)\\
	&\ge\left(1-q^{-1}-q^{7t-3-2\nu}\right)q^{3t+2-2\nu}N'(t;2t;2\nu)\\
	&>\ 0.
\end{align*}
Then $c_1(\nu,m_1,2t,t)<c_2(\nu,m_1,2t,t)$.

Suppose $m_2<2t$. We have $t\ge2$ and
\begin{equation*}
	\begin{aligned}
		c_1(\nu,m_1,m_2,t)&\le{m_2-t+1\b1}N'(t+1;m_1;2\nu)\left(N'(t;m_2;2\nu)+q^{m_2-t+1}{t\b1}\right)\\
		&\le{t\b1}N'(t+1;m_1;2\nu)\left(N'(t;m_2;2\nu)+q^{t}{t\b1}\right).
	\end{aligned}
\end{equation*}
By Lemma \ref{gaosifangsuo}, we have
\begin{equation}\label{sp-jishu-xiajie}
	N'(t;m_2;2\nu)\ge\dfrac{q^{2(\nu-m_2+1)}-1}{q-1}\ge q^{2\nu-2m_2+1}.
\end{equation}
Together with $2\nu\ge2m_1+m_2+t+1\ge2m_2+2t$, $t\ge2$, \eqref{cp3cp4-size} and Lemma\ref{gaosifangsuo}, we get
%\begin{equation*}
\begin{align*}
	&\dfrac{c_3(\nu,m_1,m_2,t)}{N'(t+1;m_1;2\nu)N'(t;m_2;2\nu)}\\
	\ge&\left({t+1\b1}-q{t\b1}\cdot\dfrac{q^{m_2-t}-1}{q^{2(\nu-t)}-1}\right)-{t\b1}\left(1+\dfrac{q^{t}{t\b1}}{N'(t;m_2;2\nu)}\right)\\
	%\ge&\ q^t-q{t\b1}\left(\dfrac{q^{m_2-t}-1}{q^{2(\nu-t)}-1}+\dfrac{q^{2t}}{q^{2\nu-2m_2+1}}\right)\\
	>&\ q^t\left(1-q^{m_2+t+1-2\nu}-q^{2m_2+2t-1-2\nu}\right)\\
	\ge&\ q^t\left(1-q^{-2m_1}-q^{-1}\right)\\
	>&\ 0,
\end{align*}
%\end{equation*}
Then $c_1(\nu,m_1,m_2,t)<c_2(\nu,m_1,m_2,t)$.
\end{proof}

\begin{lem}\label{cp-cp'}
	Let $\nu$, $m_1$, $m_2$ and $t$ be positive integers with $m_1>m_2\ge t+1$ and $2\nu\ge2m_1+m_2+2$. The following hold.
	\begin{itemize}
		\item[\rmo] $c_1(\nu,m_1,m_2,t)>c_1(\nu,m_2,m_1,t)$.
		\item[\rmt] $c_2(\nu,m_1,m_2,t)>c_2(\nu,m_2,m_1,t)$.
	\end{itemize}
\end{lem}
\begin{proof}
	(1) By Lemmas \ref{gaosifangsuo}, \ref{c1>c0} and $2\nu\ge2m_1+m_2+2\ge m_1+2m_2-t+4$, we have
	%\begin{equation*}
	\begin{align*}
		\dfrac{c_1(\nu,m_1,m_2,t)}{N'(t+1;m_1;2\nu)N'(t;m_2;2\nu)}\ge{m_2-t+1\b1}-q^{-2}=q{m_2-t\b1}+1-q^{-2}.
	\end{align*}
	%\end{equation*}
Since $m_1>m_2\ge t+1$, by \eqref{sp-jishu-xiajie}, we obtion $N'(t+1;m_1;2\nu)\ge q^{2\nu-2m_1+1}$.
Together with $2\nu\ge2m_1+m_2+2$, we get
%\begin{equation*}
\begin{align*}
	&\dfrac{c_1(\nu,m_2,m_1,t)}{N'(t+1;m_1;2\nu)N'(t;m_2;2\nu)}\\
	\le&\left(N'(t;m_1;2\nu)+q^{m_1-t+1}{t\b1}\right){m_1-t+1\b1}\dfrac{N'(t+1;m_2;2\nu)}{N'(t+1;m_1;2\nu)N'(t;m_2;2\nu)}\\
	\le&{m_1-t+1\b1}\left(\dfrac{N'(t;m_1;2\nu)N'(t+1;m_2;2\nu)}{N'(t+1;m_1;2\nu)N'(t;m_2;2\nu)}+\dfrac{q^{m_1+1}N'(t+1;m_2;2\nu)}{N'(t+1;m_1;2\nu)N'(t;m_2;2\nu)}\right)\\
	\le&{m_1-t+1\b1}\dfrac{N'(t;m_1;2\nu)N'(t+1;m_2;2\nu)}{N'(t+1;m_1;2\nu)N'(t;m_2;2\nu)}+q^{4m_1+m_2+1-4\nu}\\
	\le&{m_1-t+1\b1}\dfrac{N'(t;m_1;2\nu)N'(t+1;m_2;2\nu)}{N'(t+1;m_1;2\nu)N'(t;m_2;2\nu)}+q^{-m_2-3}.
\end{align*}
%\end{equation*}
Then by $2\nu\ge2m_1+m_2+2$ and $m_1>m_2\ge t+1$, we have
%\begin{equation*}
\begin{align*}
	&\ \dfrac{c_1(\nu,m_1,m_2,t)-c_1(\nu,m_2,m_1,t)}{N'(t+1;m_1;2\nu)N'(t;m_2;2\nu)}\\
	>&\ q{m_2-t\b1}+(1-q^{-2}-q^{-m_2-3})-{m_1-t+1\b1}\dfrac{N'(t;m_1;2\nu)N'(t+1;m_2;2\nu)}{N'(t+1;m_1;2\nu)N'(t;m_2;2\nu)}\\
	%=&\ (1-q^{-2}-q^{-5})+{m_2-t\b1}\left(q-\dfrac{q^{m_1-t+1}-1}{q^{m_1-t}-1}\right)\\
	=&\ 1-q^{-2}-q^{-m_2-3}-\dfrac{q^{m_2-t}-1}{q^{m_1-t}-1}\\
%	>&\ 1-q^{-2}-q^{-m_2-3}-q^{-1}\\
	>&\ 0.
\end{align*}
%\end{equation*}
Then the desired result follows.

(2) From \eqref{cp3cp4-size} and $m_1>m_2$, we obtain
\begin{equation*}
	\begin{aligned}
		\dfrac{c_2(\nu,m_1,m_2,t)-c_2(\nu,m_2,m_1,t)}{{t+1\b1}N'(t+1;m_1;2\nu)N'(t+1;m_2;2\nu)}&=\dfrac{N'(t;m_2;2\nu)}{N'(t+1;m_2;2\nu)}-\dfrac{N'(t;m_1;2\nu)}{N'(t+1;m_1;2\nu)}\\
		&=(q^{2(\nu-t)}-1)\left(\dfrac{1}{q^{m_2-t}-1}-\dfrac{1}{q^{m_1-t}-1}\right)\\
		&>0,
	\end{aligned}
\end{equation*}
as desired.
\end{proof}

To present the proof of Theorem \ref{non-trivial} briefly, we prove the following two lemmas. Write
$$c_0(\nu,m_1,m_2,t)=q{m_2-t\b1}N'(t+1;m_1;2\nu)N'(t;m_2;2\nu).$$
\begin{lem}\label{sp-cross-upper-other}
	Let $\nu$, $m_1$, $m_2$ and $t$ be positive integers with $m_1,m_2\ge t+1$ and $2\nu\ge m_1+m_2+1+\max\{m_1+t+2,m_2\}$. Suppose $\mf_1\subset\mp_{m_1}$ and $\mf_2\subset\mp_{m_2}$ are non-trivial cross $t$-intersecting families with $\tau_t(\mf_2)\ge\tau_t(\mf_1)$ and $(\tau_t(\mf_1),\tau_t(\mf_2))\neq(t,t+1)$. Then $|\mf_1||\mf_2|<c_1(\nu,m_1,m_2,t)$.
\end{lem}
\begin{proof}
	Suppose $\tau_t(\mf_1)=t$. By assumption, we have $m_1\ge\tau_t(\mf_2)\ge t+2$. By Lemmas \ref{gaosifangsuo}, \ref{sp-yibanshangjie}, \ref{sp-dijian} and $2\nu\ge 2m_1+m_2+t+3\ge m_1+m_2+2t+5$, we have
	\begin{equation*}
		\begin{aligned}
			\dfrac{|\mf_1||\mf_2|}{c_0(\nu,m_1,m_2,t)}&\le\dfrac{\left({\tau_t(\mf_2)\b t}{m_2-t+1\b1}^{\tau_t(\mf_2)-t}N'(\tau_t(\mf_2);m_1;2\nu)\right)N'(t;m_2;2\nu)}{q{m_2-t\b1}N'(t+1;m_1;2\nu)N'(t;m_2;2\nu)}\\
			&\le\dfrac{{t+2\b2}{m_2-t+1\b1}^2N'(t+2;m_1;2\nu)}{q{m_2-t\b1}N'(t+1;m_1;2\nu)}\\
			&=\dfrac{(q^{m_1-t-1}-1)(q^{m_2-t+1}-1)}{q(q^{2(\nu-t-1)}-1)(q^{m_2-t}-1)}{t+2\b2}{m_2-t+1\b1}\\
			&<q^{m_1+m_2+2t+5-2\nu}\\
			&\le1.
		\end{aligned}
	\end{equation*}
	Then the desired result follows from Lemma \ref{c1>c0}.

Suppose $\tau_t(\mf_1)\ge t+1$. We have $\tau_t(\mf_2)\ge t+1$ and $m_1\ge t+1$. By Lemma \ref{gaosifangsuo},  \ref{sp-yibanshangjie}, \ref{sp-dijian} and $2\nu\ge\max\{m_1+m_2+2t+4,2m_1+m_2+1,m_1+2m_2+1\}$, we obtain
%\begin{equation*}
\begin{align*}
	\dfrac{|\mf_1||\mf_2|}{c_0(\nu,m_1,m_2,t)}&\le\dfrac{{t+1\b1}^2{m_1-t+1\b1}{m_2-t+1\b1}N'(t+1;m_1;2\nu)N'(t+1;m_2;2\nu)}{q{m_2-t\b1}N'(t+1;m_1;2\nu)N'(t;m_2;2\nu)}\\
	&=\dfrac{(q^{m_2-t+1}-1)}{q(q^{2(\nu-t)}-1)}{t+1\b1}^2{m_1-t+1\b1}\\
	&<q^{m_1+m_2+2t+3-2\nu}\\
	&<1.
\end{align*}
%\end{equation*}
Then the desired result follows from Lemma \ref{c1>c0}.
\end{proof}

	\begin{lem}\label{sp-cross-upper-tt+1}
	Let $\nu$, $m_1$, $m_2$ and $t$ be positive integers with $\nu>m_i\ge t+1$, $i\in\{1,2\}$, and $2\nu\ge m_1+2m_2+3$. Suppose $\mf_1\subset\mp_{m_1}$ and $\mf_2\subset\mp_{m_2}$ are maximal non-trivial  cross $t$-intersecting families with $(\tau_t(\mf_1),\tau_t(\mf_2))=(t,t+1)$. Then one of the following holds.
	\begin{itemize}
		\item[\rmo] $\mf_1=\mmc_1(M,T;m_1,t)$ and $\mf_2=\mmc_2(M,T;m_2)$ for some $M\in\mp_{m_2+1}$ and $T\in{M\b t}$.
		\item[\rmt] $\mf_1=\mmc_3(S;m_1) $and $\mf_2=\mmc_4(S;m_2,t)$ for some $S\in\mp_{t+1}$.
		\item[\rmth] $|\mf_1||\mf_2|<c_1(\nu,m_1,m_2,t)$.
	\end{itemize}
\end{lem}
	\begin{proof}
	Let $T$ be a $t$-cover of $\mf_1$ with dimension $t$ and $\ms$ denote the set of all $t$-covers of $\mf_2$ with dimension $t+1$. By Lemma \ref{sp-cover}, each element of $\ms$ contains $T$. Let $M$ be a subspace of $\ff^{2\nu}$ generated by $\bigcup_{S\in\ms}S$. For each $F\in\mf_2\bs(\mf_2)_T$ and $S\in\ms$, we have 
	$\dim(F\cap T)=t-1$, $\dim(F\cap S)=t$ and
	$$m_2+1=\dim(T+F)\le\dim(S+F)=m_2+1.$$
	Then $T+F=S+F$, which implies that $T+F=M+F$. Thus $\dim(F\cap M)=\dim M-1$ and $t+1\le\dim M\le m_2+1$.		
	Note that the type of $M$ is $(\dim M,0)$ or $(\dim M,1)$.

Since $\tau_t(\mf_2)=t+1$, there exists $F_{2,1}\in\mf_2$ such that $T\not\subset F_{2,1}$. Write
$$H:=T\vee F_{2,1}.$$
We have
$$S=T+(S\cap F_{2,1})\subset H,\quad\dim H=m_2+1.$$
For each $F\in\mf_1$, since $T\subset F$, $T\not\subset F_{2,1}$ and $\dim(F\cap F_{2,1})\ge t$, we have
$$
	\dim(F\cap H)\ge\dim T+\dim(F\cap F_{2,1})-\dim(T\cap F_{2,1})\ge t+1.
$$
Therefore
\begin{equation}\label{sp-f1-str}
	\mf_1\subset\{F\in\mp_{m_1}: T\subset F, \dim(F\cap H)\ge t+1\}.
\end{equation}

\medskip
\noindent{\bf Case 1. $\dim M=t+1$.}
\medskip

Let $S$ be the unique member of $\ms$. Since $T\subset S$, for $F\in\mf_1$, either $S\subset F$ or $S\cap F=T$ holds.

Suppose $S$ is contained in each member of $\mf_1$. We have
$$\mf_1\subset\mmc_3(S;m_1),\quad\mf_2\subset\mmc_4(S;m_2,t).$$
Together with the maximality of $\mf_1$ and $\mf_2$, (2) holds.

Now suppose there exists $F_{1,1}\in\mf_1$ with $S\cap F_{1,1}=T$. 
Let $I\in\mm(t+1,H)$ with  $T\subset I\not\subset S$. 
	Since $I$ is not a $t$-cover of $\mf_2$, there exists $F_{2,2}\in\mf_2$ such that $\dim(I\cap F_{2,2})<t$. 
	 Note that $\dim(T\cap F_{2,2})\ge t-1$.  
	We have $\dim(I\cap F_{2,2})=t-1$. 
	Since $F_{2,2}$ is a $t$-cover of $\mf_1$, by Lemma \ref{sp-fangsuo}, we obtain
$$|(\mf_1)_I|\le{m_2-t+1\b1}N'(t+2;m_1;2\nu).$$
Note that $S\subset H$.
Then by \eqref{sp-f1-str}, we have
\begin{align*}
	|(\mf_1)_T\bs(\mf_1)_S|&\le\sum_{I\in\mm(t+1,H),\ T\subset I\not\subset S}|(\mf_1)_I|\\
	&\le\left({m_2-t+1\b1}-1\right){m_2-t+1\b1}N'(t+2;m_1;2\nu)\\
	&=q{m_2-t\b1}{m_2-t+1\b1}N'(t+2;m_1;2\nu).
\end{align*}
Together with Lemma \ref{gaosifangsuo} and $2\nu\ge m_1+2m_2-t+4\ge m_1+t+1$, we get
\begin{equation}\label{sp-cross-m=t+1-f1}
	\begin{aligned}
		|\mf_1|\le&\ |(\mf_1)_S|+|(\mf_1)_T\bs(\mf_1)_S|\\
		\le&\ N'(t+1;m_1;2\nu)+q{m_2-t\b1}{m_2-t+1\b1}N'(t+2;m_1;2\nu)\\
		=&\left(1+q{m_2-t\b1}{m_2-t+1\b1}\dfrac{q^{m_1-t-1}-1}{q^{2(\nu-t-1)}-1}\right)N'(t+1;m_1;2\nu)\\
		<&\ (1+q^{m_1+2m_2-t+3-2\nu})N'(t+1;m_1;2\nu)\\
		\le&\ (1+q^{-1})N'(t+1;m_1;2\nu).
	\end{aligned}
\end{equation}
	Let $T'\in{S\b t}\bs\{T\}$. Note that $\dim(F_{1,1}\cap T')=t-1$. Since $F_{1,1}$ is a $t$-cover of $\mf_2$, by Lemma \ref{sp-fangsuo}, we have
$$
	|(\mf_2)_{T'}\bs(\mf_2)_S|\le\left({m_1-t+1\b1}-1\right)N'(t+1;m_2;2\nu)=q{m_1-t\b1}N'(t+1;m_2;2\nu).
$$
	Then by Lemma \ref{gaosifangsuo} and $2\nu\ge m_1+m_2+t+4\ge m_2+t$, we obtain
\begin{align*}
	|\mf_2|&=|(\mf_2)_{T}|+\sum_{T'\in{S\b t}\bs\{T\}}|(\mf_2)_{T'}\bs(\mf_2)_S|\\
	&\le\ N'(t;m_2;2\nu)+q^2{t\b1}{m_1-t\b1}N'(t+1;m_2;2\nu)\\
	&=\left(1+q^2{t\b1}{m_1-t\b1}\dfrac{q^{m_2-t}-1}{q^{2(\nu-t)}-1}\right)N'(t;m_2;2\nu)\\
	&\le\ (1+q^{m_1+m_2+t+2-2\nu})N'(t;m_2;2\nu)\\
	&\le\ (1+q^{-2})N'(t;m_2;2\nu).
\end{align*}
Together with \eqref{sp-cross-m=t+1-f1} and Lemma \ref{gaosifangsuo}, we get
$$\dfrac{|\mf_1||\mf_2|}{c_0(\nu,m_1,m_2,t)}\le\dfrac{(1+q^{-1})(1+q^{-2})}{q{m_2-t\b1}}\le q^{t+1-m_2}\le1.$$
Then (3) follows from Lemma \ref{c1>c0}.

	\medskip
\noindent{\bf Case 2. $t+2\le\dim M\le m_2$.}
\medskip

Assume that there exists $F_{2,3}\in\mf_2\bs(\mf_2)_T$ with $T+F_{2,3}\neq H$. Observe that $\dim(T\vee F_{2,3})=m_2+1$ and $\dim(F\cap(T\vee F_{2,3}))\ge t+1$ for each $F\in\mf_1$. 
Set
$$W=(T+F_{2,3})\cap H,\quad w=\dim W,$$
$$\mmu=\left\{(U_1,U_2)\in{T+F_{2,3}\b t+1}\times{H\b t+1}:  T\subset U_1\not\subset W, T\subset U_2\not\subset W\right\}.$$
We have
\begin{equation}\label{sp-uu1u2}
	\mf_1\subset\left(\bigcup_{U\in{W\b t+1},\ T\subset U}(\mf_1)_U\right)\cup\left(\bigcup_{(U_1,U_2)\in\mmu}(\mf_1)_{U_1+U_2}\right).
\end{equation}
For each $(U_1,U_2)\in\mmu$, since $U_1\subset T+F_{2,3}$ and $U_1\not\subset W$, we have $U_1\not\subset H$ and $U_1\neq U_2$. Together with $T\subset U_1\cap U_2$, we get $\dim(U_1+U_2)=t+2$. 
Thus
$$\left|\bigcup_{(U_1,U_2)\in\mmu}(\mf_1)_{U_1+U_2}\right|\le\left({m_2-t+1\b1}-{w-t\b1}\right)^2N'(t+2;m_1;2\nu).$$
Then by $\left|\left\{U\in{W\b t+1}:  T\subset U\right\}\right|={w-t\b1}$ and \eqref{sp-uu1u2}, we obtain
\begin{align*}
	|\mf_1|&\le{w-t\b1}N'(t+1;m_1;2\nu)+\left({w-t\b1}-{m_2-t+1\b1}\right)^2N'(t+2;m_1;2\nu)\\
	&=x^2N'(t+2;m_1;2\nu)+xN'(t+1;m_1;2\nu)+{m_2-t+1\b1}N'(t+1;m_1;2\nu),
\end{align*}
where $x={w-t\b1}-{m_2-t+1\b1}$.
By $2\nu\ge m_1+m_2+3$ we have
$${w-t\b1}-{m_2-t+1\b1}\ge-q^{m_2-t+1}\ge-q^{2\nu-m_1-t-2}\ge-\dfrac{N'(t+1;m_1;2\nu)}{2N'(t+2;m_1;2\nu)}.$$
Together with $w\le m_2$, $N'(t+2;m_1;2\nu)\ge0$ and the property of quadratic function, we obtain
$$|\mf_1|\le{m_2-t\b1}N'(t+1;m_1;2\nu)+q^{2m_2-2t}N'(t+2;m_1;2\nu).$$
Then by Lemma \ref{gaosifangsuo} and $2\nu\ge m_1+2m_2-t+1$ we get
\begin{equation}\label{sp-cross-m>t+1-f1}
	\begin{aligned}
		|\mf_1|
		&\le\left({m_2-t\b1}+\dfrac{q^{2m_2-2t}(q^{m_1-t-1}-1)}{q^{2(\nu-t-1)}-1}\right)N'(t+1;m_1;2\nu)\\
		&<\left({m_2-t\b1}+q^{m_1+2m_2-t+1-2\nu}\right)N'(t+1;m_1;2\nu)\\
		&\le\left({m_2-t\b1}+1\right)N'(t+1;m_1;2\nu).
	\end{aligned}
\end{equation}

Set $k:=\dim M$. We have
$$\mf_2\subset\{F\in\mp_{m_2}:  T\subset F\}\cup\{F\in\mp_{m_2}:  T\not\subset F, \dim(F\cap M)=k-1\},$$
which implies that
\begin{equation*}
	|\mf_2|\le N'(t;m_2;2\nu)+q^{k-t}{t\b1}N'(k-1;m_2;2\nu).
\end{equation*}
By Lemma \ref{gaosifangsuo} and $2\nu\ge m_1+2m_2>m_2+k$, we have
$$\dfrac{q^{k-t}N'(k-1;m_2;2\nu)}{q^{k-t+1}N'(k;m_2;2\nu)}=\dfrac{q^{2(\nu-k+1)}-1}{q(q^{m_2-k+1}-1)}\ge q^{2\nu-m_2-k}\ge1,$$
which implies that
\begin{equation*}%\label{sp-cross-m>t+1-f2}
	|\mf_2|\le N'(t;m_2;2\nu)+q^{2}{t\b1}N'(t+1;m_2;2\nu).
\end{equation*}
Together with \eqref{sp-cross-m>t+1-f1}, $m_2\ge k\ge t+2$, $2\nu\ge m_1+2m_2+3$ and Lemma \ref{gaosifangsuo}, we obtain
%\begin{equation*}
\begin{align*}
	&\dfrac{c_0(\nu,m_1,m_2,t)-|\mf_1||\mf_2|}{N'(t+1;m_1;2\nu)N'(t;m_2;2\nu)}\\
	\ge&{m_2-t\b1}-1-q^2{t\b1}\left({m_2-t\b1}+1\right)\dfrac{q^{m_2-t}-1}{q^{2(\nu-t)}-1}\\
	=&{m_2-t\b1}\left(1-\dfrac{q-1}{q^{m_2-t}-1}-\dfrac{q^2(q^t-1)(q^{m_2-t}-1)}{(q-1)(q^{2(\nu-t)}-1)}-\dfrac{q^2(q^t-1)}{q^{2(\nu-t)}-1}\right)\\
	>&{m_2-t\b1}(1-q^{-1}-q^{m_2+2t+2-2\nu}-q^{3t+2-2\nu})\\
	\ge&{m_2-t\b1}(1-q^{-1}-q^{-m_1-m_2+2t-1}-q^{-m_1-2m_2+3t-1})\\
	>&{m_2-t\b1}(1-q^{-1}-q^{-3}-q^{-4})\\
	>&\ 0.
\end{align*}
%\end{equation*}
Then (3) followis from Lemma \ref{c1>c0}.

Now suppose that $T+G=H$ holds for each $G\in\mf_2\bs(\mf_2)_T$. 
We have
$$\mf_2\subset\{F\in\mp_{m_2}: T\subset F\}\cup{H\b m_2}.$$
Note that $\dim H=m_2+1$ and \eqref{sp-f1-str} holds. We will investigate cross $t$-intersecting families under this case in Case 3.

\medskip
\noindent{\bf Case 3. $\dim M=m_2+1$.}
\medskip

In this case, we have $M=H$. We have
$$\mf_1\subset\{F\in\mp_{m_1}:  T\subset F, \dim(F\cap M)\ge t+1\},\quad\mf_2\subset\{F\in\mp_{m_2}:  T\subset F\}\cup{M\b m_2}.$$
If $M$ is totally isotropic, by the maximality of $\mf_1$ and $\mf_2$, we have
$$\mf_1=\mmc_1(M,T;m_1,t),\quad\mf_2=\mmc_2(M,T;m_2),$$
i.e., (1) holds. 

Now suppose that $M$ is not totally isotropic. Note that the type of $M$ is $(m_2+1,1)$. Since $M=T+F_{2,1}$, it is routine to check that
$$|\{S\in\mm(t+1,M): T\subset S\}|={m_2-t\b1}.$$
Together with \eqref{sp-f1-str}, we have
$$|\mf_1|\le{m_2-t\b1}N'(t+1;m_1;2\nu).$$
Then by $2\nu\ge m_1+2m_2-t+2$, it is routine to check that
$$|\mf_1|<q{m_2-t\b1}N'(t+1;m_1;2\nu)\le s_0(\nu,m_1,m_2+1,t).$$
Since $M$ is of type $(m_2+1,1)$, we have
$$|\mf_2|\le N'(t;m_2;2\nu)+(q+1)<N'(t;m_2;2\nu)+q^{m_2-t+1}{t\b1}.$$
Then from Lemma \ref{c1>c0}, we obtain
$$|\mf_1||\mf_2|<s_0(\nu,m_1,m_2+1,t)\left(N'(t;m_2;2\nu)+q^{m_2-t+1}{t\b1}\right)\le c_1(\nu,m_1,m_2,t),$$
i.e., (3) holds.
	\end{proof}

	\begin{proof}[\bf Proof of Theorem \ref{non-trivial}]
		Let $\nu$, $m_1$, $m_2$ and $t$ be positive integers with $m_1\ge m_2\ge t+1$ and $2\nu\ge2m_1+m_2+t+3$. Suppose $\mf_1\subset\mp_{m_1}$  and $\mf_2\subset\mp_{m_2}$ are non-trivial cross $t$-intersecting families with maximum product of sizes. Note that
		\begin{equation*}\label{sp-fif2-c1c2}
			|\mf_1||\mf_2|\ge\max\{c_1(\nu,m_1,m_2,t),c_1(\nu,m_2,m_1,t),c_2(\nu,m_1,m_2,t),c_2(\nu,m_2,m_1,t)\}.
		\end{equation*}
	
		By Lemma \ref{sp-cross-upper-other}, we have$(\tau_t(\mf_1),\tau_t(\mf_2))=(t,t+1)$ or $(\tau_t(\mf_1),\tau_t(\mf_2))=(t+1,t)$. Together with Lemmas \ref{cp1-cp2}, \ref{cp-cp'} and \ref{sp-cross-upper-tt+1}, we finish our proof.
	\end{proof}

\end{document}